\documentclass{article}

\usepackage{amsmath}
\usepackage{amssymb}
\usepackage{amsthm}
\usepackage{eufrak}

\theoremstyle{plain}
   \newtheorem {thm}{Theorem}[section]
   
   \newtheorem {cor}[thm]{Corollary}

\theoremstyle{definition}
   \newtheorem{rmk}[thm]{Remark}

\numberwithin{equation}{section}

\newcommand{\hgs}[6]{ {}_{#1}\phi_{#2} \left[ \genfrac{}{}{0pt}{}{#3}{#4} ;
{#5},{#6} \right]}

\newcommand{\legendre}[2]{\genfrac{(}{)}{0.5pt}{}{#1}{#2}}
\newcommand{\gp}[2]{\genfrac{[}{]}{0pt}{}{#1}{#2}}

\newenvironment{oldresult}[1]
  {\textbf{#1.} \itshape } 

\allowdisplaybreaks[1]

\title
{On Series Expansions of\\Capparelli's Infinite Product}
\author{Andrew V. Sills\\
Department of Mathematics,\\
   Rutgers University, Piscataway, NJ, USA 08854-8019\\
\texttt{http://www.math.rutgers.edu/\~{}asills}\\
\texttt{asills@math.rutgers.edu}}
\date{September 9, 2003.}
\begin{document}
\maketitle

\begin{abstract} 
Using Lie theory, Stefano Capparelli conjectured an interesting 
Rogers-Ramanujan type partition identity in his 1988 Rutgers Ph.D.
thesis.  The first proof was given by George Andrews, using combinatorial
methods.  Later, Capparelli was able to provide a Lie theoretic proof.

 Most combinatorial Rogers-Ramanujan type identities (e.g. the
G\"ollnitz-Gordon identities, Gordon's combinatorial generalization
of the Rogers-Ramanujan identities, etc.) have an analytic counterpart.
The main purpose of this paper is to provide two new series representations
for the infinite product associated with
Capparelli's conjecture.  Some additional related identities, including 
new infinite families are also presented.  
\end{abstract}

\section{Introduction}
In 1894, L.J. Rogers was the first to discover a pair of
series--product identities 
which are now known
as the Rogers-Ramanujan identities.  They may be stated compactly as
follows:

\begin{oldresult}{Rogers-Ramanujan Identities--Analytic Form}
For $\lambda=0$ or $1$, 
\begin{equation}
 \sum_{j=0}^\infty \frac{q^{j^2+\lambda j}}{(q)_j} =
 \frac{1}{(q^{\lambda+1},q^{4-\lambda};q^5)_\infty},
\end{equation} 
where 
 \[ (A)_0 := (A;q)_0 := 1, \]
 \[ (A)_n := (A;q)_n := (1-A)(1-Aq) \cdots (1-Aq^{n-1}), \]
 \[ (A)_\infty:= (A;q)_\infty := \prod_{i=0}^\infty (1-Aq^i), \]
 and \[ (A_1, A_2, \dots, A_r;q)_s = (A_1;q)_s (A_2;q)_s \cdots
   (A_r,q)_s. \] 
\end{oldresult}
(Although the results in this paper may be considered purely from the
point of view of formal power series, they also yield identities
of analytic functions provided $|q|<1$.)

A \emph{partition} $\pi$ of an integer $n$ is a nonincreasing finite sequence
of positive integers $(\pi_1, \pi_2, \dots, \pi_s)$ 
such that $\sum_{i=1}^s \pi_i = n$.  The
$\pi_i$'s are called the \emph{parts} of the partition $\pi$.    

MacMahon~\cite{MacMahon} and Schur~\cite{Schur}
independently saw that the Rogers-Ramanujan identities
were in fact equivalent to the following partition theoretic statement: 

\begin{oldresult}{Rogers-Ramanujan Identities--Combinatorial Form}
Let $R_1(\lambda,n)$ denote the number of partitions 
$\pi=(\pi_1,\dots,\pi_s)$ of $n$
into parts wherein $\pi_s >\lambda$ and $\pi_{i} - \pi_{i+1} \geqq 2$.  
Let $R_2(\lambda,n)$  
denote the number of partitions of $n$ wherein all parts are congruent to
$\pm(\lambda+1)$ modulo $5$.  Then for all integers $n$ and for $\lambda=0$
or $1$,
$R_1(\lambda,n) = R_2(\lambda,n)$.
\end{oldresult} 

Over the years, many other analytic and combinatorial identities of
Rogers-Ramanujan type were discovered, including the following
analytic identity of Slater~\cite[p. 155, equations (36) and (34)]{Slater} 
and its combinatorial
counterpart due to G\"ollnitz~\cite{Gollnitz}, and rediscovered by 
Gordon~\cite{Gordon}:

\begin{oldresult}{Slater's mod 8 Identities}
For $\lambda=0$ or $1$, 
 \begin{equation}
   \sum_{j=0}^\infty \frac{ q^{j^2 + 2\lambda j} (-q;q^2)_j}{ (q^2;q^2)_j }
  = \frac{1}{(q^{1+2\lambda},q^4, q^{7-2\lambda}; q^8)_\infty}.
 \end{equation}
\end{oldresult}

\begin{oldresult}{The G\"ollnitz-Gordon Partition Identities}
Let $G_1(\lambda,n)$ denote the number of partitions $\pi=(\pi_1,\dots,\pi_s)$ of $n$ wherein 
$\pi_s > 2\lambda$, $\pi_{i} - \pi_{i-1} \geqq 2$, 
and $\pi_{i} - \pi_{i+1} > 2$ if $\pi_{i}$ or $\pi_{i+1}$ is even.
Let $G_2(\lambda,n)$ denote the number of partitions of $n$ into parts congruent
to $\pm (1+2\lambda)$ or $4$ modulo 8.  Then $G_1(\lambda,n) = 
G_2(\lambda,n)$ for all
integers $n$ and $\lambda=0$ or $1$.
\end{oldresult}

  Following a program of research initiated by Lepowsky-Milne~(\cite{LM1},
\cite{LM2}),
and Lepowsky-Wilson~(\cite{LW1},\cite{LW2},\cite{LW3},\cite{LW4},\cite{LW5}), 
Stefano Capparelli was able to conjecture a
partition identity as a result of his studies of the standard level 3
modules associated with the Lie Algebra $A_2^{(2)}$, and included
this conjecture in his Ph.D. thesis~\cite{sc-thesis}:

\begin{oldresult}{Capparelli's Conjecture}
Let $C_1(n)$ denote the number of partitions
$\pi=(\pi_1,\dots,\pi_s)$ of $n$ wherein
$\pi_s > 1$, $\pi_{i} - \pi_{i+1} \geqq 2$, and if
 $\pi_{i} - \pi_{i+1}< 4$, then either $\pi_{i}$ and $\pi_{i+1}$
 are are both multiples of three, or  $\pi_{i}\equiv 1\pmod{3}$ and
$\pi_{i+1}\equiv-1\pmod{3}$.  Let $C_2(n)$ denote the 
number of partitions of $n$ into parts congruent to $\pm 2$ or $\pm 3$
modulo $12$.  Then $C_1(n) = C_2(n)$ for all integers $n$.
\end{oldresult}
George Andrews, inspired by the combinatorial techniques of 
Wilf and Zeilberger~\cite{a=b}, provided the 
first proof in~\cite{gea}.  Later, Lie-theoretic proofs were supplied by
Tamba and Xie~\cite{tx} and by Capparelli himself~\cite{sc}.
In~\cite{mp}, Meurman and Primc embed Capparelli's conjecture in 
an infinite family of three-color partition identities. 
 
  In~\cite{aag}, Alladi, Andrews, and Gordon provided refinements to
Capparelli's conjecture along with a corresponding identity of
generating functions.  By replacing $q$ with $q^3$, and setting
$a=q^{-2}$, $b=q^{-4}$ and $c=1$ in~\cite[p. 648--9, Lemma 2(b)]{aag},
one can deduce the following analytic counterpart to Capparelli's
conjecture:
\begin{gather}\label{Cap0}
\sum_{i,j,k\geqq 0} \frac{q^{3i^2 + i + 3j^2 - j + \frac 32 k^2 + \frac 32 k
  +3ik + 3jk} (-q^3,q^3)_{i+j} }{(q^6;q^6)_i (q^6;q^6)_j (q^3;q^3)_k} 
= \frac{1}{(q^2,q^3,q^9,q^{10}; q^{12})_\infty}.
\end{gather}
 The main goal of
this paper is to present two additional analytic identities
involving the infinite product $(q^2,q^3,q^9,q^{10}; q^{12})_\infty^{-1}$,
namely 
\begin{gather}\label{Cap1}
\sum_{n=0}^\infty \sum_{j=0}^{2n} 
  \frac{q^{n^2} \legendre{n-j+1}{3}}{(q)_{2n-j} (q)_{j}}
  = \frac{1}{(q^2,q^3,q^9,q^{10}; q^{12})_\infty},
\end{gather}
where $\legendre{n}{p}$ is the Legendre symbol, and
\begin{gather} 
1+ \underset{(n,j,r)\neq(0,0,0)}{\sum_{n,j,r\geqq 0}}
 \frac{q^{3n^2 + \frac 92 r^2 + 3j^2 + 6nj + 6nr + 6rj
 -\frac 52r - j} (q^3;q^3)_{2j+r-1} (1+q^{2r+2j}) (1-q^{6r+6j})}
{(q^3;q^3)_n (q^3;q^3)_r (q^3;q^3)_j (-1;q^3)_{j+1} (q^3;q^3)_{n+2r+2j}} 
\nonumber\\ \label{Cap2}
= \frac{1}{(q^2,q^3,q^9,q^{10}; q^{12})_\infty},
\end{gather}
which will actually arise as a corollary to the following
analytic identity, an ``$a$-generalization of an analytic
counterpart of Capparelli's conjecture":

\begin{gather}
\sum_{n=0}^\infty \sum_{r=0}^\infty \sum_{j=0}^\infty
 \frac{a^{3n+2r+2j} 
   q^{3n^2 + \frac 92r^2 + 3j^2 + 6nj + 6nr + 6rj- 
 \frac 52 r - j} 
 (a^3;q^3)_{2j+r} (1+ (aq)^{2r+2j}) (1-a^3 q^{6j+6r})}
{2 (q^3;q^3)_n (q^3;q^3)_r  (q^3;q^3)_j (-a^3 q^3;q^3)_j 
(a^3,q^3)_{n+2j+2r+1}} \nonumber \\  
 =\frac{1}{(a^3 q^3;q^3)_\infty} 
 \sum_{r=0}^\infty
\frac{a^{3r} q^{3r^2} (a^3;q^3)_r (-q^3;q^3)_{r-1} (1-a^3 q^{6r}) 
    \left( (aq)^r + (aq)^{-r} \right)}
  {(q^3;q^3)_r (-a^3 q^3;q^3)_r (1-a^3) }. 
\label{aCap}
\end{gather}

In section 2, it will be revealed how identity~(\ref{Cap1}) arises
from two of the simplest possible Bailey pairs. 
Section 3 will be devoted to a derivation of the Bailey pair 
necessary to yield identity~(\ref{aCap}).  Once this is accomplished, the identities~(\ref{Cap1}) and (\ref{Cap2}) will be embedded in infinite
family of identities:
\begin{gather}
\sum_{n_1, \dots, n_k, j\geqq 0}  
  \frac{q^{N_1^2 + N_2^2 + \cdots +N_k^2} \legendre{n_k-j+1}{3}}{(q)_{2n_k-j} (q)_{n_1} (q)_{n_2} \cdots (q)_{n_k} (q)_{j}}
  = \frac{(q^{k},q^{5k},q^{6k};q^{6k}) (q^{4k},q^{8k}; q^{12k})}{(q)_\infty} \label{multiCap1},
\end{gather}
where 
$N_i = n_i + n_{i+1} + \cdots + n_k$.
\begin{gather}
1 + \underset{(n_1,\dots,n_k,r,j)\neq(0,0,\dots,0)}{\sum_{n_1,\dots,n_k,r,j
\geqq 0}}
 \frac{q^{3(M_1^2 + \cdots + M_{k}^2) + 
 \frac 32 r^2 -\frac 52r - j} (q^3;q^3)_{2j+r-1} (1+q^{2r+2j}) 
 (1-q^{6r+6j})}
{(q^3;q^3)_{n_1} \cdots (q^3;q^3)_{n_k} (q^3;q^3)_r (q^3;q^3)_j 
(-1;q^3)_{j+1} (q^3;q^3)_{n_k+2r+2j}} 
\nonumber\\ \label{multiCap2}
 = \frac{(-q^{3k-1},-q^{3k+1}, q^{6k}; q^{6k})_\infty}{(q^3;q^3)_\infty},
\end{gather}
where 
$M_i = n_i + n_{i+1} + \cdots + n_k + r + j$.  Notice that the
$k=1$ case of (\ref{multiCap2}) is equivalent to (\ref{Cap2}) since 
\[ \frac{(-q^2, -q^4, q^6; q^6)_\infty}{(q^3; q^3)_\infty} 
 = \frac{1}{(q^2,q^3,q^9,q^{10}; q^{12})_\infty}.\]

 In section 4, some related identities will be noted.  In section 5,
we conclude with some related open questions.

\section{Implications of two simple Bailey pairs}
We will require the standard machinery of Bailey's Lemma and Bailey pairs
(see \cite{wnb1}, \cite{wnb2}, \cite[Ch. 3]{qs}).
Recall that two sequences of rational functions $(\alpha_n (a,q), 
\beta_n (a,q))$
form a \emph{Bailey pair} if for all $n\geqq 0$,
\begin{equation} \label{BPdef}
  \beta_n(a,q) = \sum_{r=0}^n \frac{\alpha_r(a,q)}{(q)_{n-r} (aq)_{n+r}},
\end{equation}
and that for any Bailey pair $(\alpha_n (a,q), \beta_n (a,q))$, the identity
 \begin{equation}\label{WBL}
 \sum_{n=0}^\infty a^n q^{n^2} \beta_n (a,q) = 
\frac{1}{(aq)_\infty} \sum_{r=0}^\infty
 a^r q^{r^2} \alpha_r (a,q) 
 \end{equation}
holds (Andrews~\cite[p. 27, equation (3.33)]{qs}).

In the literature (see e.g. Andrews~\cite[section 3.5]{qs}) the implications
of a particular Bailey pair 
(often called the ``unit Bailey pair") consisting of 
an extremely simple $\beta_n$ and its corresponding $\alpha_n$ are
considered.  Here, in contrast, we consider Bailey pairs where the
$\alpha_n$'s are of an especially simple nature.

\begin{thm}\label{leftBP}
Suppose 
\[  \alpha_{n} = \left\{ 
   \begin{array}{ll}
      1, &\mbox{if $n=0$,} \\
      2, &\mbox{if $3 | n$ and $n>0$,} \\
      0, &\mbox{otherwise.}
    \end{array} \right. \] and     
    \[  \beta_{n} =  \sum_{r=-\infty}^\infty \frac{\gp{2n}{n-3r}}{(q)_{2n}}, \] 
    where the Gaussian polynomial $\gp{A}{B}$ is defined by
    \[ \gp{A}{B}:= \left\{ \begin{array}{ll} 
          (q)_A (q)_B^{-1} (q)_{A-B}^{-1}, &\mbox{if $0\leqq A\leqq B$},\\
          0,                               &\mbox{otherwise.}
          \end{array} \right. \]
Then $(\alpha_n, \beta_n ) $ form a Bailey pair.
\end{thm} 
\begin{proof} Considereing~(\ref{BPdef}) with $a=1$,
\begin{eqnarray*} 
 \beta_n  
 &=& \sum_{r=0}^n \frac{\alpha_r }{(q)_{n-r} (q)_{n+r}}\\
 &=& 1 + \sum_{r\geqq 1} \frac{2}{(q)_{n-3r} (q)_{n+3r}}\\
 &=& \sum_{r=-\infty}^\infty \frac{1}{(q)_{n-3r} (q)_{n+3r}}\\
 &=& \sum_{r=-\infty}^\infty \frac{\gp{2n}{n-3r}}{(q)_{2n}}. 
 \end{eqnarray*}
 \end{proof} 
 
 \begin{cor}\label{leftid}
 \begin{equation*}
\sum_{n=0}^{\infty}\sum_{r=-\infty}^\infty 
\frac{q^{n^2}\gp{2n}{n-3r}}{(q)_{2n}} = 
\frac{(-q^9, -q^9, q^{18} ; q^{18} )_{\infty} }{(q)_\infty }
\end{equation*}
\end{cor}

\begin{proof}
By Theorem~\ref{leftBP} and ~(\ref{WBL}) with $a=1$,
\begin{eqnarray*}
 \sum_{n=0}^{\infty}\sum_{r=-\infty}^\infty 
\frac{q^{n^2} \gp{2n}{n-3r}}{(q)_{2n}}  &=& 
 \frac{1}{(q)_\infty} \left\{ 1+\sum_{r=1}^\infty 2 q^{(3r)^2} \right\}\\
 &=& \frac{1}{(q)_\infty} \sum_{r=-\infty}^\infty q^{9r^2} \\
 &=& \frac{(-q^9, -q^9, q^{18} ; q^{18} )_{\infty} }{(q)_\infty },
\end{eqnarray*}
by Jacobi's triple product identity~\cite[p. 12, equation (1.6.1)]{gr}.
\end{proof}
 
\begin{thm}\label{rightBP}
If 
\[  \alpha_{n} = \left\{ 
   \begin{array}{ll}
      0, &\mbox{if $3 | n$,} \\
      1, &\mbox{otherwise.}
    \end{array} \right. \] and     
    \[  \beta_{n} = \sum_{r=-\infty}^\infty \frac{\gp{2n}{n-3r+1}}{(q)_{2n}} , \] then $(\alpha_n, \beta_n ) $ form a Bailey pair.
\end{thm} 

\begin{proof}
\begin{eqnarray*}
 \beta_n  
 &=& \sum_{r=0}^n \frac{\alpha_r }{(q)_{n-r} (q)_{n+r}}\\
 &=& \sum_{r=0}^{\lfloor n/3 \rfloor} \frac{1}{(q)_{n-3r-1} (q)_{n+3r+1}}
    +\sum_{r=1}^{\lfloor n/3 \rfloor} \frac{1}{(q)_{n-3r+1} (q)_{n+3r-1}}\\
 &=& \sum_{r=-\infty}^\infty \frac{1}{(q)_{n-3r+1} (q)_{n+3r-1}}\\
 &=& \sum_{r=-\infty}^\infty \frac{\gp{2n}{n-3r+1}}{(q)_{2n}}.
 \end{eqnarray*}
 \end{proof}  
 
\begin{cor}\label{rightid}
\begin{equation*}
\sum_{n=0}^{\infty}\sum_{r=-\infty}^\infty \frac{q^{n^2}
\gp{2n}{n-3r+1}}{(q)_{2n}} = 
\frac{q (-q^3, -q^{15}, q^{18} ; q^{18} )_{\infty} }{(q)_\infty }
\end{equation*}
\end{cor}

\begin{proof}
By Theorem~\ref{rightBP} and ~(\ref{WBL}) with $a=1$,
\begin{eqnarray*}
 \sum_{n=0}^{\infty}\sum_{r=-\infty}^\infty 
\frac{q^{n^2} \gp{2n}{n-3r+1}}{(q)_{2n}}  &=& 
 \frac{1}{(q)_\infty} 
 \left\{ \sum_{r=0}^\infty  q^{(3r+1)^2} 
        +\sum_{r=1}^\infty  q^{(3r-1)^2} \right\}\\
 &=& \frac{1}{(q)_\infty} \sum_{r=-\infty}^\infty q^{9r^2 + 6r + 1} \\
 &=& \frac{q(-q^3, -q^{15}, q^{18} ; q^{18} )_{\infty} }{(q)_\infty },
\end{eqnarray*}
by Jacobi's triple product identity~\cite[p. 12, equation (1.6.1)]{gr}.
\end{proof}

\begin{thm} Identity~(\ref{Cap1}) is valid.
\end{thm}
\begin{proof}
Essentially all we need to do is subtract the identity in 
Corollary~\ref{rightid} from the identity in Corollary~\ref{leftid}.
For the left hand side, observe that
\begin{gather*}
\sum_{n=0}^{\infty}\sum_{r=-\infty}^\infty 
\frac{q^{n^2} \gp{2n}{n-3r}}{(q)_{2n}} -
\sum_{n=0}^{\infty}\sum_{r=-\infty}^\infty 
\frac {q^{n^2} \gp{2n}{n-3r+1}}{(q)_{2n} }\\
= \sum_{n=0}^{\infty}\sum_{k=-\infty}^\infty 
\frac{q^{n^2} \legendre{k+1}{3} \gp{2n}{n-k}}{(q)_{2n}},
\end{gather*}
  (i.e. the inner sum on $k$ sums over the $2n$-th row of the $q$-Pascal
  triangle weighting consecutive summands in turn by the factors 
  $1$, $-1$, and $0$,)
\begin{gather*}
= \sum_{n=0}^\infty \sum_{j=0}^\infty 
 \frac{q^{n^2}  \legendre{n-j+1}{j} \gp{2n}{j}}{(q)_{2n}}
  \qquad\mbox{(by setting $j=n-k$),}\\
 = \sum_{n=0}^\infty \sum_{j=0}^\infty 
 \frac{q^{n^2}  \legendre{n-j+1}{j} }{(q)_{2n-j} (q)_j }.
\end{gather*}

For the right hand side, observe that
\begin{eqnarray*}
&& \frac{(-q^9, -q^9, q^{18} ; q^{18} )_{\infty} }{(q)_\infty }
 -q \frac{(-q^3, -q^{15}, q^{18} ; q^{18} )_{\infty} }{(q)_\infty } \\
 &=& \frac{ (q,q^5,q^6; q^6)_{\infty} (q^4,q^8; q^{12})_{\infty}}
   {(q)_\infty }\\
 && \qquad\qquad\qquad \mbox{(by the quintuple product identity
   \cite[p. 134, ex. 5.6]{gr})} \\
 &=& \frac{1}{(q^2,q^3,q^9,q^{10}; q^{12})_\infty}.
\end{eqnarray*}
\end{proof}

\begin{rmk} Identity~(\ref{Cap1}) can be rewritten as
\begin{gather}
  \sum_{n=0}^\infty \sum_{j=0}^\infty \frac{q^{n^2 + 2nj + j^2}}
    {(q)_{2j+1} (q)_{2n} } \left( \legendre{j-n+1}{3} +
     \legendre{j-n-1}{3} q^{2j+1} + \legendre{j-n}{3} q^{2n} \right)
     \nonumber\\
    = \frac{1}{(q^2,q^3,q^9,q^{10}; q^{12})_\infty}\label{altCap1}
\end{gather}
by splitting the inner sum on $j$ in the left hand side of~(\ref{Cap1})
into even and odd $j$, interchanging the order of summation and
replacing $n$ by $n+j$.  In this formulation, both sums are truly
infinite over all nonnegative $n$ and $j$.
\end{rmk}

\section{Another Bailey pair and its implications} 

Recall the standard notation for basic hypergeometric series
\[ \hgs{s+1}{s}{a_1, a_2, \dots, a_{s+1}}{b_1, b_2, \dots, b_s}{q}{z} 
:= \sum_{n=0}^\infty \frac{ (a_1)_n (a_2)_n \cdots (a_{s+1})_n} 
{(q)_n (b_1)_n (b_2)_n \dots (b_s)_n } z^n. \]

\begin{rmk}
The real challenge here was to find an appropriate $\alpha_r (a,q)$
so that 
  \begin{itemize}
    \item when $\alpha_r (a,q)$ is inserted into (\ref{BPdef}), the
    resulting expression is a (finite product multiplied by a) basic  
    hypergeometric series which
    can be transformed appropriately, and 
    \item when $\alpha_r (a,q)$ is inserted into (\ref{WBL}) and
    $a$ is set to $1$, the generating function 
      \[ \frac{1}{(q^2, q^3, q^9, q^{10}; q^{12})_{\infty}}
      =\sum_{n=0}^{\infty} C_2(n) q^n\] results.
   \end{itemize}
Once this is achieved, the power of Bailey's lemma and Bailey chains
allows us to derive a number of identities with little additional effort.
\end{rmk}

\begin{thm} \label{BP} If 
  \begin{equation}
   \alpha_r (a,q) := \frac{(a)_r (1-aq^{2r}) (-q)_{r-1}}{(q)_r (1-a)
    (-aq)_r}\left( (aq)^{r/3} + (aq)^{-r/3} \right) 
  \end{equation}
    and  
  \begin{equation}
    \beta_n (a,q) := 
    \sum_{j=0}^n \sum_{r=0}^{n-j} \frac{a^{-(r+j)/3} 
        q^{\frac 12 r^2 - \frac 56 r -\frac 13 j} (1-aq^{2j+2r}) 
         (a)_{2j+r} (1+ (aq)^{\frac 23(r+j)}) }
          { 2 (q)_j (q)_r (a)_{n+j+r+1} (-aq)_j (q)_{n-j-r}},
  \end{equation}
then $\big(\alpha_n (a,q) , \beta_n (a,q) \big)$ form a Bailey pair.
\end{thm}

\begin{proof} 
 \begin{eqnarray*}
&& \beta_n(a,q)\\
 &=& \sum_{r=0}^n \frac{\alpha_r(a,q)}{(q)_{n-r} (aq)_{n+r}} \\
           &=& \frac{1}{(q)_{n} (aq)_n}
     \sum_{r=0}^n \frac{(-1)^r q^{nr-\frac 12 r^2 + \frac 12 r}}
{(aq^{n+1})_r} \alpha_r (a,q)\\
   &=& \frac{1}{(q)_{n} (aq)_n}
     \sum_{r=0}^n \frac{(-1)^r q^{nr-\frac 12 r^2 + \frac 12 r}}
{(aq^{n+1})_r} \frac{(a)_r (1-aq^{2r}) (-q)_{r-1}}{(q)_r (1-a)
    (-aq)_r}\left( (aq)^{r/3} + (aq)^{-r/3} \right) \\
    &=& \frac{1}{(q)_{n} (aq)_n}
     \sum_{r=0}^n \frac{(-1)^r q^{nr-\frac 12 r^2 + \frac 12 r}}
{(aq^{n+1})_r} \frac{(a)_r (aq^2;q^2)_r (-1)_{r}}{2(q)_r (a;q^2)_r
    (-aq)_r}\left( (aq)^{r/3} + (aq)^{-r/3} \right) \\
  &=& \frac{1}{2(q)_n (aq)_n} \lim_{\tau\to 0}
   \left( \hgs{6}{5}{a,q\sqrt{a},-q\sqrt{a},-1,\tau a q, q^{-n}}
             {\sqrt{a}, -\sqrt{a}, -aq, \frac{1}{\tau}, a q^{n+1}}{q}
             {\tau^{-1} a^{\frac 13} q^{n+\frac 13}}\right.\\
  && \qquad\qquad\qquad\qquad
  + \left. \hgs{6}{5}{a,q\sqrt{a},-q\sqrt{a},-1,\tau a q, q^{-n}}
              {\sqrt{a},-\sqrt{a}, -aq,\frac{1}{\tau}, a q^{n+1}}{q}
              {\tau^{-1} a^{-\frac 13} q^{n-\frac 13}} 
   \right)\\
   &=& \frac{1}{2(q)_n (aq)_n} \lim_{\tau\to 0}
   \left( \sum_{j=0}^{n} \frac{ (-\frac 1\tau)_{j} (q\sqrt{a})_{j} 
      (-q\sqrt{a})_j (-1)_{j} (\tau aq)_{j} (q^{-n})_{j} (a)_{2j}}
      {(q)_j (\sqrt{a})_j (\sqrt{a})_j (-aq)_j (\frac 1\tau)_j 
       (aq^{n+1})_{j} } a^{j/3} q^{nj + \frac 56 j - \frac 12 j^2} \right. \\
       &&\qquad\qquad\qquad\qquad
       \times 
       \hgs{4}{3}{aq^{2j},q^{j+1}\sqrt{a},-q^{j+1}\sqrt{a}, q^{j-n}}
       {q^j\sqrt{a},-q^j\sqrt{a},aq^{n+j+1}}{q}{-a^{\frac 13}
       q^{n-j+\frac 13}}\\
      &&\qquad\qquad\qquad
      +\sum_{j=0}^{n} \frac{ (-\frac 1\tau)_{j} (q\sqrt{a})_{j} 
      (-q\sqrt{a})_j (-1)_{j} (\tau aq)_{j} (q^{-n})_{j} (a)_{2j}}
      {(q)_j (\sqrt{a})_j (\sqrt{a})_j (-aq)_j (\frac 1\tau)_j 
       (aq^{n+1})_j} a^{-j/3} q^{nj + \frac 16 j - \frac 12 j^2}  \\
       &&\qquad\qquad\qquad\qquad
       \times\left. 
       \hgs{4}{3}{aq^{2j},q^{j+1}\sqrt{a},-q^{j+1}\sqrt{a}, q^{j-n}}
       {q^j\sqrt{a},-q^j\sqrt{a},aq^{n+j+1}}{q}{-a^{-\frac 13} q^{n-j-\frac 13}}
       \right)\\
       &&\qquad\qquad\mbox{(by~\cite[p. 34, equation (2.4.1)]{gr})}\\
       &=& \frac{1}{2(q)_n (aq)_n} \sum_{j=0}^n \sum_{r=0}^{n-j}
       \frac{ (aq^2;q^2)_{j+r} (q^{-n})_{j+r} (a)_{2j+r}
       (1+ (aq)^{\frac 23(j+r)} ) }
        {(q)_j (q)_r (a;q^2)_{j+r} (-aq)_j (aq^{n+1})_{j+r} }\\
        &&\qquad\qquad\qquad\qquad\times
        (-1)^{j+r} a^{-\frac 13 j -\frac 13 r}
        q^{nj + nr +\frac 16 j - \frac 12 j^2 - \frac 13 r - jr}\\
        &=&\sum_{j=0}^n \sum_{r=0}^{n-j} \frac{a^{-(r+j)/3} 
        q^{\frac 12 r^2 - \frac 56 r -\frac 13 j} (aq^2;q^2)_{j+r} 
         (a)_{2j+r} (1+ (aq)^{\frac 23(r+j)}) }
          { 2 (q)_j (q)_r (aq)_{n+j+r} (a;q^2)_{j+r} (-aq)_j (q)_{n-j-r}}\\
        &=&\sum_{j=0}^n \sum_{r=0}^{n-j} \frac{a^{-(r+j)/3} 
        q^{\frac 12 r^2 - \frac 56 r -\frac 13 j} (1-aq^{2j+2r}) 
         (a)_{2j+r} (1+ (aq)^{\frac 23(r+j)}) }
          { 2 (q)_j (q)_r (a)_{n+j+r+1} (-aq)_j (q)_{n-j-r}}.
\end{eqnarray*}
\end{proof}

\begin{thm} Identity~(\ref{aCap}) is valid.
\end{thm}
\begin{proof}
Insert the Bailey pair in Theorem~\ref{BP} into  
equation~(\ref{WBL}), and then replace $a$ by $a^3$ and
$q$ by $q^3$ throughout.  On the left hand side, interchange the
order of summation bringing $j$ out in front of $n$ 
and replace $n$ by $n+j$.  Then, interchange the order of summation
bringing $r$ in front of $n$ and replace $n$ by $n+r$.
\end{proof} 

\begin{rmk} Andrews~\cite{geapc} pointed out that a direct proof
(i.e. one that is independent of Bailey's lemma) of~(\ref{aCap}) is possible.
Start out with the left hand side of~(\ref{aCap}) and
set $t=r+j$ so that the double sum is now on $r$ and $t$.  The 
inner sum on $r$ is 
  \[ \hgs{2}{1}{-a^{-3} q^{-3t}, q^{-3t}}{a^{-3} q^{3-6t}}{q^3}{q^3} \]
 which is summable by the
$q$-Chu-Vandermonde formula~\cite[p. 236, equation (II.6)]{gr}.
This form can then be converted to the right hand side using
a formula due to Euler~\cite[p. 19, equation (2.2.5)]{top}.
 \end{rmk}

\begin{cor} Identity~(\ref{Cap2}) is valid.
\end{cor}
\begin{proof} Set $a=1$ in identity~(\ref{aCap}), then apply Jacobi's
triple product identity~\cite[p. 63, (7.1)]{qs} to the right hand side,
and simplify the resulting product.
\end{proof}

 Now that Bailey pairs have been established, it is a simple matter to
embed the analytic Capparelli identities into infinite families of
identities using the notion of 
the ``Bailey chain"~(\cite{multiRR},~\cite[p. 28 ff.]{qs}):
\begin{thm} Identity~(\ref{multiCap1}) is valid.
\end{thm}
\begin{proof} Insert the Bailey pairs from Theorem~\ref{leftBP}
and Theorem~\ref{rightBP} into
equation (3.34) of Andrews~\cite[p. 30]{qs}, with $a=1$.  Subtract 
the second equation from the first, interchange orders of summation
and change summation variables as appropriate.
\end{proof}

\begin{thm}
Identity~(\ref{multiCap2}) is valid.
\end{thm}
\begin{proof} Insert $\alpha_n (a,q)$ and $\beta_n(a,q)$ into
equation (3.34) of Andrews~\cite[p. 30]{qs}, interchange orders of summation
and change summation variables as appropriate. Finally, replace $a$ by
$a^3$ and $q$ by $q^3$ throughout.
\end{proof}

\section{Related identities}
In order to obtain the $a$-generalization of the analytic counterpart
of Capparelli's conjecture, the Bailey pair from Theorem~\ref{BP}
was inserted into equation (\ref{WBL}), which is a limiting case of
Bailey's lemma~\cite[p. 25, Thm. 3.3]{qs}.  
We now require a different limiting case of Bailey's lemma:
\hfil\break
\emph{If $(\alpha_r (a,q), \beta_j (a,q))$ form a Bailey pair, then}
\begin{gather} 
   \sum_{n\geqq 0} a^n q^{n^2} (-q;q^2)_n \beta_n (a,q^2) 
  =\frac{(-aq;q^2)_\infty}{(aq^2;q^2)_\infty} 
   \sum_{r=0}^\infty \frac{ a^r q^{r^2} (-q;q^2)_r} {(-aq;q^2)_r}
    \alpha_r(a,q^2). \label{ATNSBL}
  \end{gather}

Now, inserting the Bailey pair from Theorem~\ref{BP} into (\ref{ATNSBL}), we obtain the identity
\begin{gather}
\sum_{n=0}^\infty \sum_{r=0}^\infty \sum_{j=0}^\infty
 \frac{a^{3n+2r+2j} 
   q^{3n^2 + 6 r^2 + 3j^2 + 6nj + 6nr + 6rj- 5r - 2j} 
 (-q^3;q^6)_{n+j+r} (a^3;q^3)_{2j+r} }
{2 (q^6;q^6)_n (q^6;q^6)_r  (q^6;q^6)_j (-a^3 q^6;q^6)_j 
(a^3,q^6)_{n+2j+2r+1}} \nonumber \\ 
 \times (1+ a^{2r+2j} q^{4r+4j}) (1-a^3 q^{12j+12r}) \nonumber\\
 =\frac{(-a^3 q^3)_\infty}{(a^3 q^6;q^6)_\infty} 
 \sum_{r=0}^\infty
\frac{a^{3r} q^{3r^2} (-q^3;q^6)_r (a^3;q^6)_r (-q^6;q^6)_{r-1} (1-a^3 q^{12r})  \left( (aq^2)^r + (aq^2)^{-r} \right)}
  {(q^6;q^6)_r (-a^3 q^6;q^6)_r (1-a^3) }, 
\label{aATNSCap}
\end{gather}
which, for $a=1$, yields
\begin{gather}
1+\underset{(n,j,r)\neq (0,0,0)}{\sum_{n=0}^\infty \sum_{r=0}^\infty \sum_{j=0}^\infty}
 \frac{ q^{3n^2 + 6 r^2 + 3j^2 + 6nj + 6nr + 6rj- 5r - 2j} 
 (-q^3;q^6)_{n+j+r} (q^3;q^3)_{2j+r-1} (1+ q^{4r+4j}) 
 (1-q^{12j+12r})}
{ (q^6;q^6)_n (q^6;q^6)_r  (q^6;q^6)_j (-1;q^6)_{j+1} 
(q^6;q^6)_{n+2j+2r}} \nonumber \\  
 = (-q;q^2)_{\infty} .
\label{ATNSCap}
\end{gather}
Note the extremely simple product on the right hand side of 
(\ref{ATNSCap}), which is the generating function for partitions
into distinct odd parts.

The analogous identity relative to (\ref{Cap1}) is
\begin{gather}\label{ATNSCap1}
\sum_{n=0}^\infty \sum_{j=0}^{2n} 
  \frac{q^{n^2} (-q;q^2)_n \legendre{n-j+1}{3}}{(q^2;q^2)_{2n-j} (q^2;q^2)_{j}}
  = \frac{(q^6;q^{12})_{\infty}}{(q^3,q^9 ; q^{12})_\infty}.
\end{gather}

Of course, (\ref{aATNSCap}), (\ref{ATNSCap}), and (\ref{ATNSCap1}) could easily be embedded
in infinite families of identities via the Bailey chain.

\section{Conclusion}
While we now have in hand two series representations for the
infinite product
${(q^2,q^3,q^9,q^{10}; q^{12})_\infty}^{-1},$
namely the left hand sides of~(\ref{Cap1}) and~(\ref{Cap2}), it is
not clear exactly how the partitions $C_1 (n) $ are generated
by them.  Such an explanation would be most welcome.

Also, it should be noted that this infinite product
${(q^2,q^3,q^9,q^{10}; q^{12})_\infty}^{-1}$ has appeared
in the literature in at least two other combinatorial contexts
beside Capparelli's conjecture: see Andrews' \emph{Memoir} on
generalized Frobenius partitions~\cite[p. 10, equation (5.9)]{geamem},
and Propp's paper on generalized Ferrers diagrams~\cite[p. 113, Thm. 4(a)]{jp},
although in both of these cases the product contained the additional
factor $(q)_\infty^{-1}$.
It would be interesting to see a direct connection between
Capparelli's $C_1(n)$ partitions and the combinatorial
constructs of Andrews and Propp.

\section*{Acknowledgement}
I thank George Andrews for suggesting that I look at this problem, and for his encouragement during the project.


\begin{thebibliography}{99}
\bibitem{aag} K. Alladi, G.E. Andrews, and B. Gordon, Refinements and
generalizations of Capparelli's conjecture on partitions, J. Algebra
174 (1995), 636--658.

\bibitem{top} G.E. Andrews, \textit{The Theory of Partitions},
Encyclopedia of Mathematics and its Applications, vol. 2, 
Addison-Welsey, 1976.  Reissued Cambridge Univ. Press, 1998.

\bibitem{multiRR} G.E. Andrews, Multiple series Rogers-Ramanujan type
identities, Pacific J. Math 114 (1984), 267--283.

\bibitem{geamem} G.E. Andrews, Generalized Frobenius partitions,
Mem. Amer. Math. Soc. 49 (1984), no. 301.

\bibitem{qs}
G.E. Andrews, \textit{q-series: their development and application in
analysis, number theory, combinatorics, physics, and computer algebra},
CBMS Regional Conference Series in Mathematics, no. 66, 
Amer. Math. Soc., Providence, 1986.

\bibitem{gea}
G.E. Andrews, Schur's Theorem, Capparelli's conjecture and $q$-trinomial
coefficients, in ``Proc. Rademacher Centenary Conf., 1992'' pp. 141--154,
Contemp. Math. 166, Amer. Math. Soc., Providence, 1994.

\bibitem{geapc} G.E. Andrews, private communication, August 18, 2003.


\bibitem{wnb1} W.N. Bailey, Some identities in combinatory analysis,
Proc. London Math Soc. (2), 49 (1947), 421--425.

\bibitem{wnb2} W.N. Bailey, On identities of the Rogers-Ramanujan type,
Proc. London Math Soc. (2), 50 (1948), 1--10. 

\bibitem{sc-thesis} S. Capparelli, Vertex operator relations for affine
Lie algebras and combinatorial identities, Ph.D. thesis, Rutgers, 1988.

\bibitem{sc} S. Capparelli, A construction of the level 3 modules
for the affine algebra $A_2^{(2)}$ and a new combinatorial identity of
the Rogers-Ramanujan type, Trans. Amer. Math. Soc., 348 (1996), no. 2,
481--501.

\bibitem{gr} G. Gasper and M. Rahman, \textit{Basic Hypergeometric
Series}, Cambridge Univ. Press, 1990.

\bibitem{Gollnitz} H. G\"ollnitz, Einfache partitionen, Dipomarbeit W.S.,
G\"ottingen, 1960.

\bibitem{Gordon} B. Gordon, Some continued fractions of the Rogers-Ramanujan
type, Duke J. 31(1965), 741--748.

\bibitem{LM1} J. Lepowsky and S.C. Milne, Lie algebraic approaches to
classical partition identities, \textit{Adv. in Math.} 29 (1978), no. 1,
15--59.

\bibitem{LM2} J. Lepowsky and S.C. Milne, Lie algebras and classical
partition identities, \textit{Proc. Nat. Acad. Sci. USA} 75 (1978), no. 2,
578--579.

\bibitem{LW1} J. Lepowsky and R.L. Wilson, Construction of the
affine Lie algebra $A_1^{(1)}$, \textit{Commun. Math. Phys.} 62 (1978),
43--53.

\bibitem{LW2} J. Lepowsky and R.L. Wilson, The Rogers-Ramanujan identities:
Lie theoretic interpretation and proof, \textit{Proc. Natl. Acad. Sci USA}
78 (1981), 699--701.

\bibitem{LW3} J. Lepowsky and R.L. Wilson, A Lie theoretic interpretation
and proof of the Rogers-Ramanujan identities, \textit{Advances in Math.}
45 (1982), 21--72.

\bibitem{LW4} J. Lepowsky and R.L. Wilson, A new family of algebras underlying
the Rogers-Ramanujan identities and generalizations, \textit{Invent. Math.}
77 (1984) 199--290.

\bibitem{LW5} J. Lepowsky and R.L. Wilson, The structure of standard
modules, I: Universal algebras and the Rogers-Ramanujan identities,
\textit{Invent. Math.} 79 (1985), 417--442.

\bibitem{MacMahon} P.A. MacMahon, 
\emph{Combinatory Analysis}, vol. 2, Cambridge Univ. Press,
1918.

\bibitem{mp} A. Meurman and M. Primc, A basis of the basic 
$\EuFrak{sl}(3,\mathbb{C})\sp \sim$-module,
\textit{Commun. Contemp. Math.} 3 (2001), no. 4, 593--614.

\bibitem{a=b} M. Petkov\v{s}ek, H. Wilf, and D. Zeilberger, \textit{A=B},
A.K. Peters, 1996.

\bibitem{jp} J. Propp, Some variants of Ferrers diagrams, J. Comb. Theory A,
52 (1989), 98--128.

\bibitem{ljr} L.J. Rogers, Second memoir on the expansion of certain 
infinite products, Proc. London Math. Soc. 25(1894), 318--343.

\bibitem{Schur} I. Schur, Ein Beitrag zur additeven Zahlentheorie und zur 
Theorie der Kettenbr\"uche, Sitzungsberichte der Berliner Akademie, 1917,
302--321.

\bibitem{Slater}
L.J. Slater, Further Identities of the Rogers-Ramanujan Type,
Proc. London Math. Soc. (2), 54 (1952), 147--167.

\bibitem{tx} M. Tamba and C. Xie, Level three standard modules for
$A_2^{(2)}$ and combinatorial identities, J. Pure Appl. Algebra 105 (1995),
no. 1, 53--92.

\end{thebibliography}
\end{document}